\documentclass[10pt,a4paper,twoside]{article}

\usepackage{amsmath,amsfonts,amssymb,latexsym,wasysym,mathrsfs,multirow,pifont,dsfont}
\usepackage{amsthm}
\usepackage[english]{babel}
\usepackage{a4wide}
\usepackage{pstricks,graphicx}
\usepackage[scriptsize]{subfigure}
\usepackage[scriptsize,bf]{caption}
\usepackage{enumitem}
\usepackage[small,euler-digits,icomma,OT1,T1]{eulervm}
\usepackage[all,2cell]{xy} \UseAllTwocells \SilentMatrices
\usepackage{hyperref}

\theoremstyle{plain}
\newtheorem{theorem}{Theorem}[section]
\newtheorem{proposition}{Proposition}[section]
\newtheorem{lemma}{Lemma}[section]

\theoremstyle{definition}

\theoremstyle{remark}
\newtheorem{remark}{Remark}[section]

\title{Macroscopic limit of a bipartite Curie-Weiss model: a dynamical approach}

\author{Francesca Collet\\
{\small Dipartimento di Matematica} \\[-0.2cm]
{\small Alma Mater Studiorum Universit\`a di Bologna}\\[-0.2cm]
{\small Piazza di Porta San Donato 5; 40126 Bologna, Italy} \\[-0.2cm]
{\small e-mail: francesca.collet@unibo.it}}

\date{}

\begin{document}

\maketitle

\begin{abstract}
\noindent We analyze the Glauber dynamics for a bi-populated Curie-Weiss model. We obtain the limiting behavior of the empirical averages in the limit of infinitely many particles. We then characterize the phase space of the model in absence of magnetic field and we show that several phase transitions in the inter-groups interaction strength occur. 

\vspace{0.3cm}

\end{abstract}

\section{Introduction}

Theoretical models based on mean field interacting spin systems, although simplistic, are able to show a good qualitative description of cooperative macroscopic behavior in self-organizing systems. In the last decades, for this reason and their analytical tractability, they have also been applied to social sciences \cite{BrDu01, CoDaPSa10}, finance \cite{FrBa08, DPRST09}, chemistry \cite{DBAgBaBu12} and ecology \cite{VoBaHuMa03, BDPFFM14}.\\
An interesting family within this class, which has naturally emerged in applications, is a multi-species extension of the Curie-Weiss model. The possibility of taking account for several kinds of magnetic spins is a peculiar feature, that may be relevant to capture diverse phenomena from magnetism in anisotropic materials \cite{ScDa97} to socio-economic models \cite{CoGh07}. In this last case, for instance, the possibility of partitioning a population in classes of people sharing the same characteristics yields to a more realistic modeling. In this respect, it is worth to mention a series of papers \cite{CoGaMe08, GaBaCo09}, where this kind of interacting particle systems and techniques coming from statistical mechanics are used not only to describe problems concerning cultural coexistence, immigration and integration in societies, but also to make predictions at the population level. This framework has been tested against experimental data with an overall good agreement \cite{BCSV14}. \\
In this paper, the system under consideration consists of a two-population generalization of the classical mean field Ising model. Thus, on the complete graph two types of spins are present, with an interaction which is homogeneous among sites belonging to the same group, whereas it assumes a different value between sites of distinct populations. To have more intuition it may be helpful to think of the model as an interacting mixture of two Curie-Weiss models. \\
In the equilibrium theory, the thermodynamic limit of the model has been rigorously obtained in \cite{GaCo08}, where also some preliminary results concerning the phase diagram are given. Under mild assumptions on the interaction parameters and the restriction to populations of the same size, an analytic description of the latter is provided in \cite{FeUn12}. A variety of phase transitions is shown.\\
Here, however, we examine in detail some aspects of the dynamics of the model. We start with a Glauber dynamics for the $N$-particle system, where each spin may experience a flip with a rate depending on the gradient of the Hamiltonian felt by the particle. Next we show, via a large deviation approach on path space, that a propagation of chaos result holds and establish the asymptotic dynamics of the pair of group magnetizations in the infinite volume limit. This result corresponds to a law of large numbers and it allows to describe the macroscopic evolution of the system, which is deterministic. We then  give the full phase diagram of the stationary solutions in absence of external fields. In comparison with the results in \cite{FeUn12}, we keep complete arbitrariness of the interaction parameters and, moreover, we relax the symmetry hypothesis on the relative proportion of the groups.

Although the model itself is not a novelty in literature and its static analysis is well understood, to our knowledge we provide the first attempt of dynamical description of the time evolution of such system, obtaining also some non-equilibrium properties. \\ 
The outline of the paper is as follows. In Section~\ref{sct:Model-Results} we illustrate the model and exact statements of the results are presented. Section~\ref{sct:Proofs} is devoted to the proofs of the main theorems stated in Section~\ref{sct:Model-Results}. A Conclusions section concludes the paper.

\section{Model and main results}\label{sct:Model-Results}

Let $\mathscr{S}=\{-1,+1\}$ and $h=(h_j)_{j=1}^N \in \mathbb{R}^N$ be a sequence of real numbers. Given a configuration $\sigma=(\sigma_j)_{j=1}^N \in \mathscr{S}^N$ of a $N$-spin system and a realization of the magnetic field $h$, we can define the Hamiltonian $H_N(\sigma):  \mathscr{S}^{N}   \longrightarrow  \mathbb{R}$ as
\begin{equation}\label{Hamiltonian}   
  H_N(\sigma)=-\frac{1}{2N}\sum_{j,k=1}^N J_{jk} \sigma_j \sigma_k - \sum_{j=1}^N h_j \sigma_j  \,,
\end{equation}
where $\sigma_j$ is the spin value at site $j$ and $h_j$ the local magnetic field associated with the same site. Let $J_{jk}$, real parameter, represent the strength of the interaction between sites $j$ and $k$. Without loss of generality, we restrict to the case of symmetric matrices $\mathbb{J} = (J_{jk})_{1 \leq j,k \leq N}$. \\
We divide the whole system of size $N$ into two disjoint subsystems of sizes $N_1$ and $N_2$ respectively. Let $I_1$ (resp. $I_2$) be the set of sites belonging to the first (resp. second) subsystem. We have $\vert I_1 \vert = N_1$ and $\vert I_2 \vert = N_2$, with $N_1 + N_2 = N$. To fix notation, let $1, 2, \dots, N_1$ be the indeces corresponding to particles in population $I_1$ and $N_1+1, N_1+2, \dots, N$ those of particles in population $I_2$, so that
\[ 
\begin{array}{rc|c}
& \mbox{\scriptsize Population $I_1$} & \mbox{\scriptsize Population $I_2$} \\
\sigma = & ( \sigma_1, \sigma_2, \dots, \sigma_{N_1} & \sigma_{N_1+1}, \sigma_{N_1+2}, \dots, \sigma_{N} ) 
\end{array}
\] 
Given two spins $\sigma_j$ and $\sigma_k$, their mutual interaction $J_{jk}$ depends on the subsystems they belong to, as specified by the following matrix
\[
\mathbb{J} =
\begin{array}{rccc}
\left(
\begin{array}{c|ccc}
\begin{smallmatrix} \mathbb{J}_{11} \end{smallmatrix} & & \begin{smallmatrix} \mathbb{J}_{12} \end{smallmatrix} \\
\hline
&&&\\
\begin{smallmatrix} \mathbb{J}_{21} \end{smallmatrix} & & \begin{smallmatrix} \mathbb{J}_{22} \end{smallmatrix} & \\
&&&
\end{array}
\right) 
\end{array}
\]
where $\mathbb{J}_{jj}$ is a $N_j \times N_j$ square block comprised by constant elements $J_{jj}$ tuning the interaction within sites of the same subsystem and $\mathbb{J}_{jk} = \mathbb{J}^{\top}_{kj}$ is a $N_j \times N_k$ rectangular block with constant entries $J_{jk}$ controlling the interaction of spins located in different subystems. We assume $J_{11}$ and $J_{22}$ positive; while, $J_{12}$ can be either positive or negative allowing both ferromagnetic and antiferromagnetic interactions.\\
Analogously, the field $h_j$ can take only the two values $h_1$ or $h_2$, depending on the subset containing $\sigma_j$, according to the vector
\[
h =
\left(
\begin{array}{c}
h_1 \\
\vdots\\
h_1\\
\hline
h_2\\
\vdots\\
h_2
\end{array}
\right) 
\hspace{-0.5cm}
\begin{array}{ll}
\left.
\begin{array}{l}
\\
\phantom{\vdots}\\
\\
\end{array}
\right\}
& N_1 \\
\left.
\begin{array}{l}
\\
\phantom{\vdots}\\
\\
\end{array}
\right\}
& N_2
\end{array} 
\]
We introduce the magnetization of a subset $S$ as
\[
m_{\vert S \vert} (\sigma) = \frac{1}{\vert S \vert} \sum_{j \in S} \sigma_j \,.
\]
Moreover, we denote by $\mathbf{m}_N(\sigma) = \left( m_{N_1} (\sigma), m_{N_2} (\sigma) \right)$ the vector whose entries are the magnetizations of population $I_1$ and $I_2$, respectively. By defining $\alpha:= \frac{N_1}{N}$ the proportion of sites belonging to the first group, we can immediately rewrite the Hamiltonian \eqref{Hamiltonian} as
\begin{multline}\label{Hamiltonian:bipop}
H_N(\sigma) = - \frac{N}{2} \left[ \alpha^2 J_{11} \left( m_{N_1} (\sigma) \right)^2 + 2 \alpha (1-\alpha) J_{12} m_{N_1} (\sigma) m_{N_2} (\sigma) + (1-\alpha)^2 J_{22} \left( m_{N_2} (\sigma) \right)^2\right] \\
- N \alpha h_1 m_{N_1} (\sigma) - N (1-\alpha) h_2 m_{N_2} (\sigma)\,.
\end{multline}

Let us define the dynamics we consider. The stochastic process $(\sigma(t))_{t \geq 0}$ is described as follows. Let $\sigma^i$ denote the configuration obtained from $\sigma$ by flipping the $i$-th spin. The spins will be assumed to evolve with Glauber one spin-flip dynamics: at any time $t$, the system may experience a transition 
\[
\sigma \longrightarrow \sigma^i 
\quad \mbox{ at rate } \quad 
\left\{
\begin{array}{ll}
 e^{-\sigma_i \left[ R_1 \left(  \mathbf{m}_N(\sigma) \right) + h_1 \right]}, & \mbox{ if } i \in I_1\\
e^{-\sigma_i \left[ R_2 \left(  \mathbf{m}_N(\sigma) \right) + h_2\right]}, & \mbox{ if } i \in I_2 \,,
\end{array}
\right.
\] 
where
\begin{equation}\label{Def:R}
R_1 \left(  \mathbf{x} \right) = \alpha J_{11} x_1 + (1-\alpha) J_{12} x_2 \quad \mbox{ and } \quad R_2 \left(  \mathbf{x} \right) = \alpha J_{12} x_1 + (1-\alpha) J_{22} x_2 \,.
\end{equation}
Not to clutter our notation, we omit subscripts indicating the dependence of functions $R_i$'s ($i=1,2$) on $\alpha$ and $\mathbb{J}$.\\

Formally, we are considering a continuous time Markov chain on $\mathscr{S}^N$, with infinitesimal generator $L_N$ acting on functions $f:\mathscr{S}^N \longrightarrow \mathbb{R}$ as follows:
\begin{equation}\label{Generator:Micro}
 L_Nf(\sigma)=\sum_{j \in I_1} e^{-\sigma_j \left[ R_1 \left(  \mathbf{m}_N(\sigma) \right) + h_1 \right]}\nabla^{\sigma}_j f(\sigma) 
+ \sum_{j \in I_2} e^{-\sigma_j \left[ R_2 \left(  \mathbf{m}_N(\sigma) \right) + h_2\right]}\nabla^{\sigma}_j  f(\sigma),
\end{equation}
where $\nabla^{\sigma}_j f(\sigma)=f(\sigma^j)-f(\sigma)$.\\

\begin{remark}
The dynamics \eqref{Generator:Micro} is reversible with respect to the stationary distribution 
\begin{equation}\label{Stat:Distr}
\mu_N(\sigma) = \frac{\exp[-H_N(\sigma)]}{Z_N} \,,
\end{equation}
with $Z_N$ normalizing factor depending on $\alpha$, $\mathbb{J}$ and $h$.
\end{remark}

For simplicity, the initial condition $\sigma(0)$ is assumed to have product distribution $\lambda^{\otimes N}$, with $\lambda$ probability measure on $\mathscr{S}$. The quantity $(\sigma_j(t))_{t \in [0,T]}$ represents the time evolution on $[0,T]$ of $j$-th spin value. The space of all these paths is $\mathscr{D}[0,T]$, which is the space of the right-continuous, piecewise-constant functions from $[0,T]$ to $\mathscr{S}$. We endow $\mathscr{D}[0,T]$ with the Skorohod topology, which provides a metric and a Borel $\sigma$-field (see~\cite{EtKu86}).


\paragraph{Result I: Infinite Volume Dynamics.} We now derive the dynamics of the process \eqref{Generator:Micro}, in the limit as $N \longrightarrow +\infty$, in a fixed time interval $[0,T]$, via a large deviations approach. 

\begin{remark}
Notice that, since $N_1 = \alpha N$ and $N_2 = (1-\alpha) N$, in the limit as $N \longrightarrow +\infty$, $N_1$ and $N_2$ grow to infinity with the same speed as $N$.
\end{remark}

Let $\sigma[0,T] \in (\mathscr{D}[0,T])^N$ denote a path of the system in the time interval $[0,T]$, with $T$ positive and fixed. If $f: \mathscr{S} \longrightarrow \mathbb{R}$, we are interested in the asymptotic (as $N \longrightarrow +\infty$) behavior of the pair of \emph{empirical averages} 
\[ 
\left( \int f d\rho_1 (t), \int f d\rho_2 (t) \right) := \left( \frac{1}{N_1} \sum_{j \in I_1} f(\sigma_j(t)), \frac{1}{N_2} \sum_{j \in I_2} f(\sigma_j(t)) \right) \,,
\]
where $\boldsymbol{\rho}_N (t) := (\rho_1 (t), \rho_2 (t))_{t \in [0,T]}$ is the flow of the pair of \emph{empirical measures}
\[ 
(\rho_1 (t), \rho_2 (t)) = \left( \frac{1}{N_1} \sum_{j \in I_1} \delta_{\sigma_j (t)}, \frac{1}{N_2} \sum_{j \in I_2} \delta_{\sigma_j (t)} \right) \,.
\]
We may think of $\boldsymbol{\rho}_N$ as a random element of $\mathscr{M}_1(\mathscr{D}[0,T]) \times \mathscr{M}_1(\mathscr{D}[0,T])$, where $\mathscr{M}_1(\mathscr{D}[0,T])$ is the space of probability measures on $\mathscr{D}[0,T]$ endowed with the weak convergence topology. \\
First of all, we want to determine the weak limit of $\boldsymbol{\rho}_N$ in $\mathscr{M}_1(\mathscr{D}[0,T]) \times \mathscr{M}_1(\mathscr{D}[0,T])$ as $N$ grows to infinity. It corresponds to a law of large numbers with the limit being a deterministic measure. Being an element of $\mathscr{M}_1(\mathscr{D}[0,T]) \times \mathscr{M}_1(\mathscr{D}[0,T])$, such a limit can be viewed as a stochastic process, which describes the dynamics of the system in the infinite volume limit.\\
The result we actually prove is a \emph{large deviation principle} for the distribution of $\boldsymbol{\rho}_N$. To state it properly we need some more notation.\\

Consider $\boldsymbol{Q} = (Q_1,Q_2) \in \mathscr{M}_1(\mathscr{D}[0,T]) \times \mathscr{M}_1(\mathscr{D}[0,T])$, if $\boldsymbol{Q}(t) = (Q_1(t), Q_2(t))$ indicates the marginal distribution of $\boldsymbol{Q}$ at time $t$, then we denote
\[
\mathbf{m}_{\boldsymbol{Q}(t)} ( \eta ) = \left( \int_{\mathscr{S}} \eta \, Q_1 (t; d\eta), \int_{\mathscr{S}} \eta \, Q_2 (t; d\eta)\right) \,.
\]
We will often use this notation in the rest of the paper. Besides, for a given path $\eta[0,T] \in \mathscr{D}[0,T]$ (observe that it is a single spin trajectory), we define 
\begin{equation}\label{F}
F(\boldsymbol{Q}) := \alpha F_1 (Q_1) + (1-\alpha) F_2 (Q_2)
\end{equation}
with
\begin{subequations}\label{F1-2}
\begin{multline}
F_1(Q_1) := \int  \bigg\{  \int_0^T  \left( 1 - e^{-\eta(t) \left[ R_1 \left( \mathbf{m}_{\boldsymbol{Q}(t)} (\eta) \right) + h_1\right]}\right)dt + \frac{\eta(T)}{2} \left[ R_1 \left( \mathbf{m}_{\boldsymbol{Q}(T)} (\eta) \right) + h_1 \right] \\
- \frac{\eta(0)}{2} \left[ R_1 \left( \mathbf{m}_{\boldsymbol{Q}(0)} (\eta) \right) + h_1 \right] \bigg\} dQ_1 
\end{multline}
and
\begin{multline}
F_2(Q_2) := \int  \bigg\{  \int_0^T  \left( 1 - e^{-\eta(t) \left[ R_2 \left( \mathbf{m}_{\boldsymbol{Q}(t)} (\eta) \right) + h_2\right]}\right)dt + \frac{\eta(T)}{2} \left[ R_2 \left( \mathbf{m}_{\boldsymbol{Q}(T)} (\eta) \right) + h_2 \right] \\
- \frac{\eta(0)}{2} \left[ R_2 \left( \mathbf{m}_{\boldsymbol{Q}(0)} (\eta) \right) + h_2 \right] \bigg\} dQ_2  \,.
\end{multline}
\end{subequations}
If $P_N$ is the law of $\sigma[0,T] \in (\mathscr{D}[0,T])^N$, the process with infinitesimal generator \eqref{Generator:Micro} and initial distribution $\lambda^{\otimes N}$, then let $\mathcal{P}_N$ be the law of $\boldsymbol{\rho}_N (\sigma[0,T])$ under $P_N$, i.e. $\mathcal{P}_{N} (\cdot) := P_{N} \{ \boldsymbol{\rho}_N \in \cdot \}$. \\
A special case is when all the sites are independent from each other (absence of interaction) and the spin sign change with constant rate equal to $1$. We denote by $W$ the marginal law on $\mathscr{D}[0,T]$ of this process.\\
In what follows, for every $\pi_1, \pi_2 \in \mathscr{M}_1(\mathscr{D}[0,T])$, the quantity
\[
H(\pi_1 \vert \pi_2) := \left \{
\begin{array}{ll}
\int d\pi_1 \log \frac{d\pi_1}{d\pi_2} & \mbox{if} \quad \pi_1 \ll \pi_2 \quad \mbox{and} \quad \log \frac{d\pi_1}{d\pi_2} \in L^1(\pi_1)\\
+\infty & \mbox{otherwise}
\end{array}
\right.
\]
will represent the relative entropy between $\pi_1$ and $\pi_2$. 

\begin{theorem}\label{thm:LDP}
The laws $\{\mathcal{P}_N\}_{N \geq 1}$ of $\boldsymbol{\rho}_N$ obey a Large Deviation Principle with good rate function 
\begin{equation}\label{RateFct_PN}
\mathcal{I}(\boldsymbol{Q}) := \alpha H(Q_1 \vert W) + (1-\alpha) H(Q_2 \vert W) - F(\boldsymbol{Q})  \,.
\end{equation}
\end{theorem}

The proof of Theorem \ref{thm:LDP}, as well as of the other results stated in this section, is postponed to Section \ref{sct:Proofs}. \\
We recall that the above statement means that, for every $C, O \subseteq \mathscr{M}_1(\mathscr{D}[0,T]) \times \mathscr{M}_1(\mathscr{D}[0,T])$ respectively closed and open for the weak topology, we have
 \begin{subequations}
\begin{equation}\label{LDPc}
\limsup_{N \to +\infty} \frac{1}{N} \log \mathcal{P}_N (C) \leq - \inf_{\boldsymbol{Q} \in C} \mathcal{I} (\boldsymbol{Q})
\end{equation}
\begin{equation}\label{LDPo}
\liminf_{N \to +\infty} \frac{1}{N} \log \mathcal{P}_N (O) \geq - \inf_{\boldsymbol{Q} \in O} \mathcal{I} (\boldsymbol{Q}) \,.
\end{equation} 
\end{subequations}
Besides, the function $\mathcal{I} ( \cdot )$ is nonnegative, lower semi-continuous and the level sets $\left\{ \boldsymbol{Q} : \mathcal{I} ( \boldsymbol{Q} ) \leq k \right\}$ are compact, for every positive $k$. See~\cite{DeZe93} for more details. \\
The large deviation principle allows to characterize the unique limit $\boldsymbol{Q}^*$ of the sequence $\left\{ \boldsymbol{\rho}_N \right\}_{N \geq 1}$ and, in particular, makes possible to provide a Fokker-Planck equation useful to describe the time evolution of such limiting probability measure.

\begin{theorem}\label{thm:MKV}
Suppose that the initial distribution of the Markov process $(\sigma(t))_{t\geq 0}$ with generator \eqref{Generator:Micro} is such that the random variables $(\sigma_j(0))_{j=1}^{N}$ are independent and identically distributed with law $\lambda$. Then the equation $\mathcal{I}(\boldsymbol{Q})=0$ admits a unique solution $\boldsymbol{Q}^* \in \mathscr{M}_1(\mathscr{D}[0,T]) \times \mathscr{M}_1(\mathscr{D}[0,T])$, such that its time marginal $\boldsymbol{q}(t) := \boldsymbol{Q}^*(t) = \left( q_1(t), q_2(t) \right) \in \mathscr{M}_1(\mathscr{S}) \times \mathscr{M}_1(\mathscr{S})$ have components which are weak solutions of the nonlinear McKean-Vlasov equation
\begin{equation}\label{MKV1}
    \begin{array}{l}
        \partial_t q_i(t) = \mathcal{L}_i q_i(t) \\
         q_i(0) = \lambda 
    \end{array}
\qquad (t \in [0,T]; i=1,2)
\end{equation}
where
\[
\mathcal{L}_i q_i (t; \eta) = \nabla^{\eta} \left\{ e^{-\eta \left[ R_i \left( \mathbf{m}_{\boldsymbol{q}(t)}(\eta) \right) + h_i \right]} q_i (t; \eta)  \right\}  \,.
\]
with $\eta \in \mathscr{S}$ and the $R_i$'s defined by \eqref{Def:R}.
Moreover, with respect to a metric $d(\cdot, \cdot)$ inducing the weak topology, $\boldsymbol{\rho}_N \longrightarrow  \boldsymbol{Q}^*$ in probability with exponential rate, i.e. $\mathcal{P}_N \left\{ d \left( \boldsymbol{\rho}_N,\boldsymbol{Q}^* \right) > \varepsilon \right\}$ is exponentially small in $N$, for each $\varepsilon>0$.
\end{theorem}


\paragraph{Result II: Phase Diagram.} The equation \eqref{MKV1} describes the behavior of the system governed by generator \eqref{Generator:Micro} in the infinite volume limit. We are interested in the detection of the \hbox{$t$-stationary} solutions of this equation and in the analysis of their stability to get a full understanding of the phase diagram.\\ 
First of all, it is convenient to reformulate the McKean-Vlasov equation \eqref{MKV1} in terms of the expectations $m_{q_1(t)}$ and $m_{q_2(t)}$ defined as follows:
\[
 m_{q_i(t)} := \sum_{\eta \in \mathscr{S}} \eta \, q_i ( t; \eta ) \quad \mbox{ for } i = 1, 2.
\]
In the sequel, by abuse of notation we write $m_i(t)$ instead of $m_{q_i(t)}$, for $i = 1, 2$, and we denote by $\mathbf{m}(t) = \left( m_1(t), m_2(t) \right)$.\\
To rewrite \eqref{MKV1} in terms of the new variables $m_1(t)$ and $m_2(t)$, observe that
\[
\dot{m}_1(t) =   \sum_{\eta \in \mathscr{S}} \eta \, \dot{q}_1 (t; \eta) = \sum_{\eta \in \mathscr{S}} \eta \, \mathcal{L}_1 q_1 (t; \eta) \,.
\]
On the other hand, a straightforward computation yields that
\[
\sum_{\eta \in \mathscr{S}} \eta \, \mathcal{L}_1 q_1 (t; \eta) = 
2 \sinh \left[ R_1 \left( \mathbf{m}(t) \right) + h_1\right] - 2 m_1(t) \cosh \left[R_1 \left( \mathbf{m}(t) \right) + h_1 \right]\,,
\]
giving
\[
\dot{m}_1(t) = 2 \sinh \left[ R_1 \left( \mathbf{m}(t) \right) + h_1\right] - 2 m_1(t) \cosh \left[R_1 \left( \mathbf{m}(t) \right) + h_1 \right] \,.
\]
Similarly, we can obtain an equation for $m_2(t)$, so that it is proved that \eqref{MKV1} can be rewritten as
\begin{equation}\label{MKV2}
\begin{array}{lcl}
\dot{m}_1(t) & = & 2 \sinh \left[ R_1 \left( \mathbf{m}(t) \right) + h_1 \right] - 2 m_1(t) \cosh \left[ R_1 \left( \mathbf{m}(t) \right) + h_1\right] \\     
\dot{m}_2(t) & = & 2 \sinh \left[ R_2 \left( \mathbf{m}(t) \right) + h_2\right] - 2 m_2(t) \cosh \left[ R_2 \left(\mathbf{m}(t) \right) + h_2\right] \,,
\end{array}
\end{equation}
with initial condition $m_1(0) = m_2(0)= m_{\lambda}$.


The following theorem is concerned with equilibria of equation \eqref{MKV2}, in the case of absence of magnetic fields. In particular, it is shown that the system undergoes several phase transitions depending on the parameters.

\begin{theorem}\label{thm:PhDia}
Consider the system \eqref{MKV2} with $h_1=h_2=0$ and fix $\alpha > \frac{1}{2}$. Moreover, set 
\[
J_{12}^{(c)} := \sqrt{\frac{(1-\alpha J_{11})(1-(1-\alpha) J_{22})}{\alpha (1-\alpha)}} \,.
\]
\begin{enumerate}
\item  
If $\alpha J_{11} \leq 1$ and $(1-\alpha) J_{22} \leq 1$, then
\begin{enumerate}
\item 
for $\vert J_{12} \vert \leq J_{12}^{(c)}$ the equation \eqref{MKV2} admits a unique equilibrium solution $\mathbf{m}^{(0)}$. If $\vert J_{12} \vert < J_{12}^{(c)}$ it is linearly stable, while if $\vert J_{12} \vert = J_{12}^{(c)}$ the linearized system has a neutral direction. 
\item 
for $\vert J_{12} \vert > J_{12}^{(c)}$ the point $\mathbf{m}^{(0)}$ is still an equilibrium for \eqref{MKV2}, but it is a saddle point for the linearized system. Moreover, \eqref{MKV2} has two linearly stable stationary solutions $\pm \mathbf{m}^{(1)}$.
\end{enumerate}
\item 
If $\alpha J_{11} \leq 1$ and $(1-\alpha) J_{22} > 1$ (or analogously $\alpha J_{11} > 1$ and $(1-\alpha) J_{22} \leq 1$), then for every value of $J_{12}$ there are three stationary solutions of \eqref{MKV2}: $\mathbf{m}^{(0)}$, which is a saddle point for the linearized system, and two linearly stable equilibria $\pm \mathbf{m}^{(1)}$.
\item 
If $\alpha J_{11} > 1$ and $(1-\alpha) J_{22} > 1$, then there exists another critical value $\widetilde{J}_{12}^{(c)}$, with $0 < \widetilde{J}_{12}^{(c)} < J_{12}^{(c)}$, such that
\begin{enumerate}
\item 
for $\vert J_{12} \vert < \widetilde{J}_{12}^{(c)}$ the stationary solutions of \eqref{MKV2} are nine: $\mathbf{m}^{(0)}$, which is linearly unstable; $\pm \mathbf{m}^{(1)}$ and $\pm \mathbf{m}^{(2)}$, which are linearly stable; $\pm \mathbf{s}^{(1)}$ and $\pm \mathbf{s}^{(2)}$, that are saddle points for the linearized system.
\item 
for $\widetilde{J}_{12}^{(c)} < \vert J_{12} \vert < J_{12}^{(c)}$ the critical points of \eqref{MKV2} reduce to five: $\mathbf{m}^{(0)}$, which is still linearly unstable; $\pm \mathbf{m}^{(1)}$, which are linearly stable and  $\pm \mathbf{s}^{(1)}$, that are saddle points for the linearized system.
\item 
for $\vert J_{12} \vert \geq J_{12}^{(c)}$ the equilibria of \eqref{MKV2} become three: $\mathbf{m}^{(0)}$ and $\pm \mathbf{m}^{(1)}$. The points $\pm \mathbf{m}^{(1)}$ are  linearly stable; $\mathbf{m}^{(0)}$ is a saddle point if $\vert J_{12} \vert > J_{12}^{(c)}$, while if $\vert J_{12} \vert = J_{12}^{(c)}$ the linearized system has a neutral direction. 
\end{enumerate}
\end{enumerate}
In all ranges of the parameters $\mathbf{m}^{(0)} = \boldsymbol{0}$; whereas, the coordinates of the points $\pm \mathbf{m}^{(1)}$, $\pm \mathbf{m}^{(2)}$, $\pm \mathbf{s}^{(1)}$ and $\pm \mathbf{s}^{(2)}$ depend on $\alpha$ and $\mathbb{J}$.
\end{theorem}

The rich scenario depicted in Theorem \ref{thm:PhDia} can be qualitatively summarized in the phase diagrams presented in Figure \ref{fig:1}, where the phase space $(\alpha J_{11}, J_{12})$ is described for a fixed value of $(1-\alpha) J_{22}$. A totally analogous picture can be drawn in the converse situation when we fix $\alpha J_{11}$ instead. 

\begin{figure}[h]
\centering%
\subfigure{\includegraphics[width=0.49\textwidth]{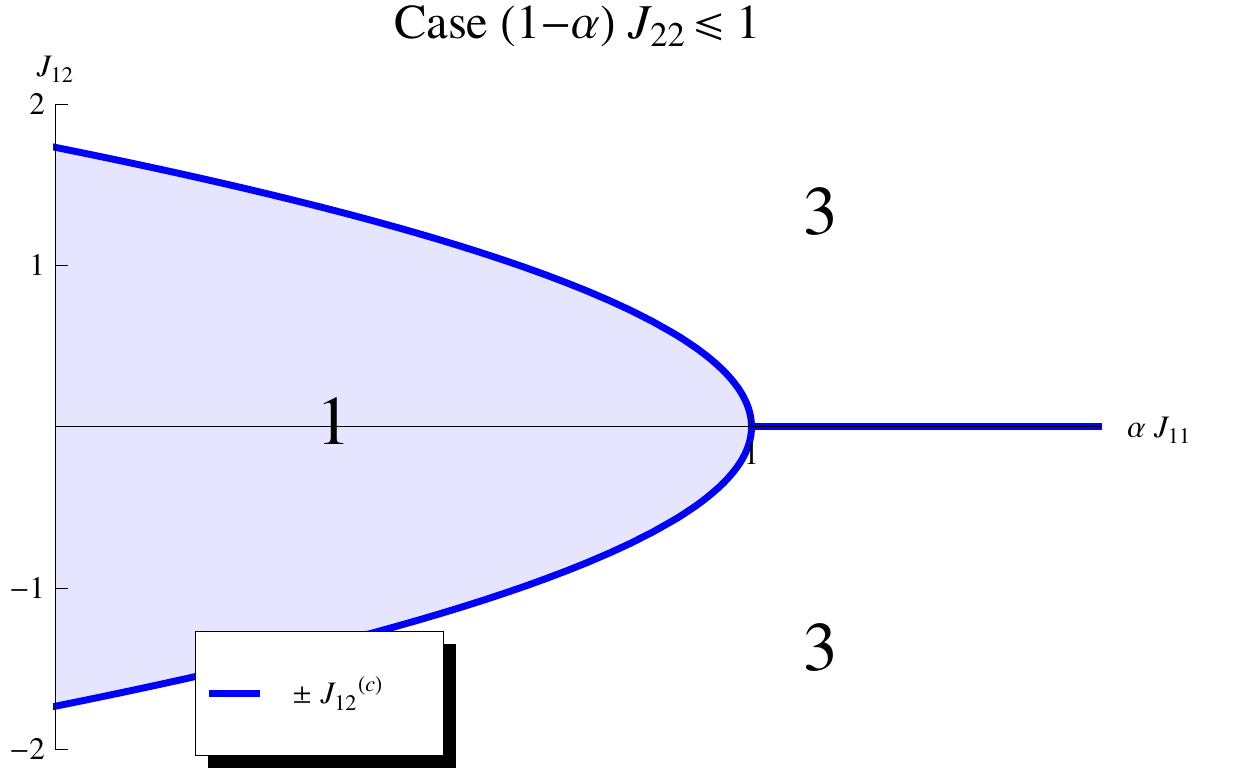}} 
\subfigure{\includegraphics[width=0.49\textwidth]{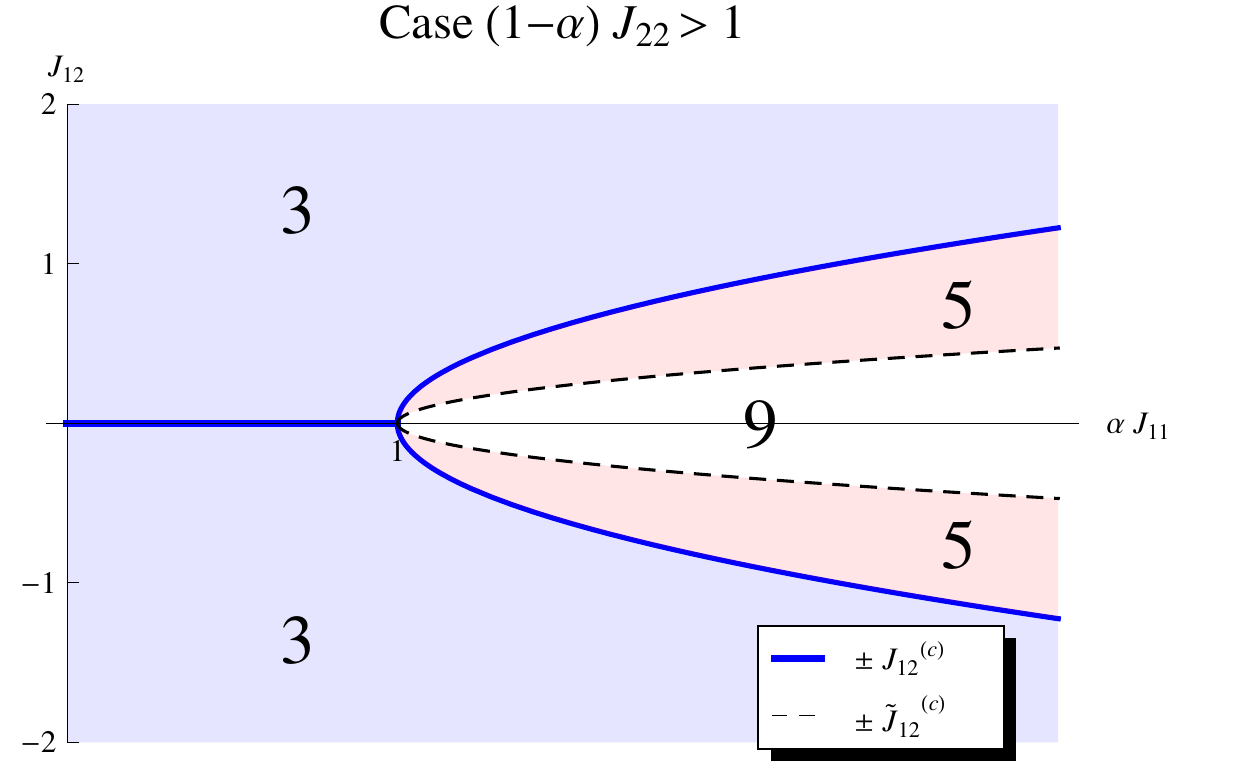}}
\caption{Picture of the phase space $\left(\alpha J_{11}, J_{12} \right)$ for a fixed value of $(1-\alpha) J_{22}$. Each coloured region represents a phase with as many equilibria of \eqref{MKV2} as indicated by the numerical label. In the left panel the case $(1-\alpha) J_{22} \leq 1$ is illustrated. The blue separation line is $\vert J_{12} \vert = J_{12}^{(c)}$. The right panel displays the case $(1-\alpha) J_{22} > 1$. The blue curve is still $\vert J_{12} \vert = J_{12}^{(c)}$, while the dashed black one corresponds to $\vert J_{12} \vert = \widetilde{J}_{12}^{(c)}$ and it is qualitative. Indeed, the latter is defined implicitly by a tangency relation and it can be obtained numerically. %
More hints about this curve will be given in the proof of Theorem \ref{thm:PhDia} in Section \ref{sct:Proofs}.}
\label{fig:1}
\end{figure}

\begin{remark}
Imagine to start with two independent Curie-Weiss models ($J_{12} =0$): one of size $N_1$ and with inverse temperature $\alpha J_{11} \leq 1$ and the other of size $N_2$ and with inverse temperature $(1-\alpha) J_{22} > 1$. The long time behavior of the whole system is well known in this independence case: the first subsystem will converge to a disordered state,  while the second one will polarize. Item \emph{(b)} in Theorem~\ref{thm:PhDia}  says that if we turn on the coupling strength $J_{12}$, whatever its intensity is, both populations will polarize in the long run. Hence, the supercritical Curie-Weiss model creates a sort of bulk of polarization able to influence also the rest of the global population.    
\end{remark}

\begin{remark}
If we consider the Hamiltonian \eqref{Hamiltonian:bipop} with $h_1=h_2=0$, the associated asymptotic free energy is given by
\begin{equation}\label{free:energy}
\mathcal{F} ( \boldsymbol{\nu} ) = - \frac{1}{2} \left[ \alpha^2 J_{11} \nu_1^2 + 2 \alpha (1-\alpha) J_{12} \nu_1 \nu_2 + (1-\alpha)^2 J_{22} \nu_2^2\right] + \alpha I^{(B)} (\nu_1) + (1-\alpha) I^{(B)} (\nu_2) \,,
\end{equation}
where $\boldsymbol{\nu} \in [-1,+1]^2$ is a magnetization vector  and the Cram{\'e}r entropy of a Bernoulli random variable $I^{(B)}$ is defined as
\[
I^{(B)} (\nu) = \frac{1-\nu}{2} \log \frac{1-\nu}{2} + \frac{1+\nu}{2} \log \frac{1+\nu}{2} \,.
\]
Rephrasing the statement of Theorem \ref{thm:PhDia} in the terminology of statics analysis we get the classification of the critical points of the functional \eqref{free:energy}. In particular, 
\begin{itemize}
\item the equilibria of \eqref{MKV2} and the critical points of \eqref{free:energy} coincide;
\item being linearly stable (resp. linearly unstable) for the dynamics \eqref{MKV2} translates into being a local minimum (resp. local maximum) for the free energy;
\item saddles are saddles in both cases. 
\end{itemize}
Moreover, it can be readily seen that, whenever they exist as critical points, $\pm \mathbf{m}^{(1)}$ are the global minima for \eqref{free:energy}. In Figure \ref{fig:2} some contour plots of the free energy surface are shown. 
\end{remark}

\setcounter{subfigure}{0}
\begin{figure}
\centering%
\subfigure[Case $\alpha J_{11} \leq 1$ and $(1-\alpha) J_{22} \leq 1$.]{
\includegraphics[width=0.4\textwidth]{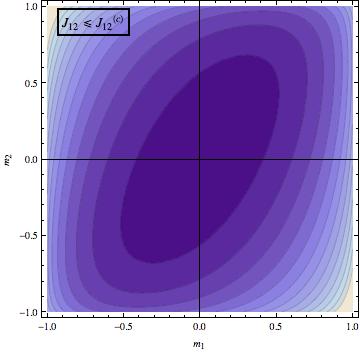} 
\includegraphics[ width=0.4\textwidth]{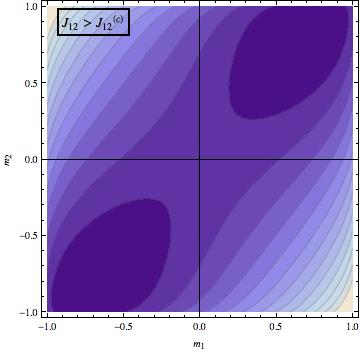} 
}
\subfigure[Case $\alpha J_{11} \leq 1$ and $(1-\alpha) J_{22} > 1$ (left panel) or, conversely, case $\alpha J_{11} > 1$ and $(1-\alpha) J_{22} \leq 1$ (right panel).]{
\includegraphics[width=0.4\textwidth]{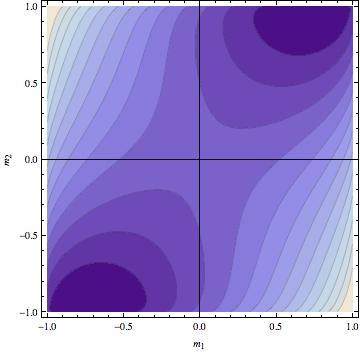} 
\includegraphics[ width=0.4\textwidth]{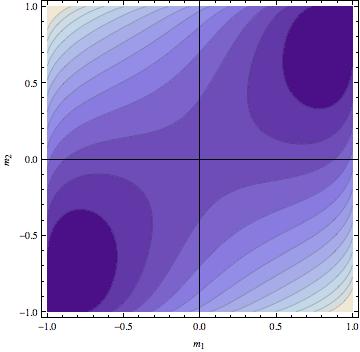} 
}
\centerline{
\subfigure[Case $\alpha J_{11} > 1$ and $(1-\alpha) J_{22} > 1$.]{
\includegraphics[width=0.4\textwidth]{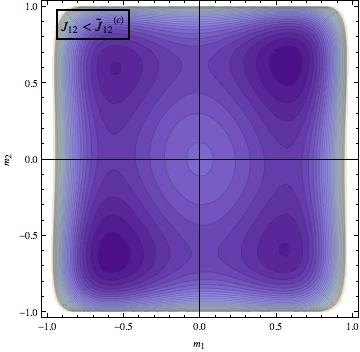} 
\includegraphics[ width=0.4\textwidth]{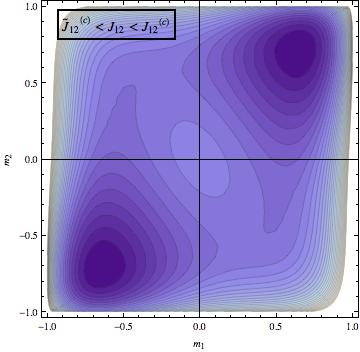} 
\includegraphics[ width=0.4\textwidth]{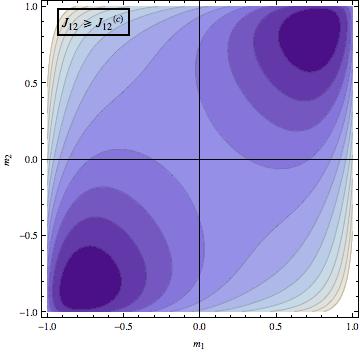} 
}}
\caption{Contour plots of the free energy surface \eqref{free:energy} for several values of the parameters. Color scale convention: the darker the shaded region is, the lower the elevation of the surface is.}
\label{fig:2}
\end{figure}

\section{Proofs}\label{sct:Proofs}

\subsection{Proof of Theorem \ref{thm:LDP}}

We start with a technical lemma that gives a representation of $P_{N}$ (law of the process) in terms of the pair of empirical measures $\boldsymbol{\rho}_N$.

\begin{lemma}\label{lmm:RND_PW}
It holds
\begin{equation}\label{RND_PW:Formula}
\frac{dP_{N}}{dW^{\otimes N}} \left( \sigma[0,T] \right) = \exp \left[ NF \left ( \boldsymbol{\rho}_N \left( \sigma[0,T] \right) \right) + O(1) \right]
\end{equation}
where, for $\boldsymbol{Q} \in \mathscr{M}_1(\mathscr{D}[0,T]) \times \mathscr{M}_1(\mathscr{D}[0,T])$, $F(\boldsymbol{Q})$ is expressed by \eqref{F}.
\end{lemma}

\begin{proof}
The proof is basically an application of the analogous of the Girsanov's formula in the case of stochastic integrals with respect to point processes (see ~\cite{Bre81, LiSh01}). \\
The exponent in the Radon-Nikodym derivative \eqref{RND_PW:Formula} consists of two analogous terms: one related to the sum over sites belonging to population $I_1$ and the other to the sum over sites in $I_2$. Not to be redundant we make explicit only the computations related to the first term; at each step the second is similar and will be referred to as {\scshape term 2}. \\
Let $(\mathcal{N}_t^{(1)}(j))_{j \in I_1}$ (resp. $(\mathcal{N}_t^{(2)}(j))_{j \in I_2}$) 
be the multivariate Poisson process counting the jumps of $\sigma_j$ in population $I_1$ (resp. $I_2$). 
If we read $\displaystyle{\sigma_j(t-) = \lim_{s \rightarrow t-} \sigma_j(s)}$ and $\displaystyle{m_{\rho_1(t-)}(\sigma) =  \lim_{s \rightarrow t-} m_{\rho_1(s)}(\sigma)}$, it yields
\begin{multline*}
\frac{dP_{N}}{dW^{\otimes N}} \left( \sigma[0,T] \right) = \exp \Bigg\{ \sum_{j \in I_1} \left[ \int_0^T  \left(1-e^{- \sigma_j(t) \left[ R_1 \left( \mathbf{m}_{\boldsymbol{\rho}_N(t)}(\sigma) \right) + h_1 \right]} \right) dt  \right. \\
\left. + \int_0^T \log e^{- \sigma_j(t-) \left[ R_1 \left( \mathbf{m}_{\boldsymbol{\rho}_N(t-)}(\sigma) \right) + h_1 \right]} d\mathcal{N}_t^{(1)}(j) \right]+ \mbox{{\scshape term 2}} \Bigg\} \,.
%
\end{multline*}
Moreover, since $\sigma$ has no simultaneous jumps and $\int \mathcal{N}^{(1)}_T d\rho_1 < +\infty$ (resp. $\int \mathcal{N}^{(2)}_T d\rho_1 < +\infty$) almost surely with respect to $W^{\otimes N}$, we get
%
%
\begin{multline*}
\frac{dP_{N}}{dW^{\otimes N}} \left( \sigma[0,T] \right) = \exp \Bigg\{ \sum_{j \in I_1} \left[ \int_0^T  \left(1 - e^{-\sigma_j(t) \left[ R_1 \left( \mathbf{m}_{\boldsymbol{\rho}_N(t)}(\sigma) \right) + h_1 \right]} \right) dt \right. \\
\left. + \int_0^T  \sigma_j(t) \left[ R_1 \left( \mathbf{m}_{\boldsymbol{\rho}_N(t)}(\sigma) \right) + h_1 \right] d\mathcal{N}_t^{(1)}(j) + O \left( \frac{1}{N_1} \right)  \right]  + \mbox{{\scshape term 2}} \Bigg\} \,.
\end{multline*}
%
Lastly, due to a general result for reversible spin-flip systems (we refer to ~\cite{DaPdHo96}, Lemma~3), we can write
\begin{multline*}
\frac{dP_{N}}{dW^{\otimes N}} \left( \sigma[0,T] \right) = \exp \Bigg\{ \sum_{j \in I_1} \left[ \int_0^T  \left(1 - e^{-\sigma_j(t) \left[ R_1 \left( \mathbf{m}_{\boldsymbol{\rho}_N(t)}(\sigma) \right) + h_1 \right]} \right) dt  \right.  \\
\left.   - \frac{\sigma_j(0)}{2} \left[ R_1 \left( \mathbf{m}_{\boldsymbol{\rho}_N(0)}(\sigma) \right) + h_1 \right] + \frac{\sigma_j(T)}{2} \left[ R_1 \left( \mathbf{m}_{\boldsymbol{\rho}_N(T)}(\sigma) \right) + h_1 \right] \right] + \mbox{{\scshape term 2}} + O(1) \Bigg\} 
%
\end{multline*}
and this leads us to the expression \eqref{F} for $F$. 
\end{proof}

Before proving Theorem \ref{thm:LDP}, we need to show the following result.

\begin{proposition}\label{prop:LDP_RN}
Let $\mathcal{W}_N ( \cdot ) := W^{\otimes N} \left\{ \boldsymbol{\rho}_N \in \cdot \right\}$ be the law of $\boldsymbol{\rho}_N$ in the case of independence. The sequence $\{\mathcal{W}_N\}_{N \geq 1}$ obeys a Large Deviation Principle (LDP) with rate function 
\begin{equation}\label{RateFct_RN}
I(\boldsymbol{Q}) = \alpha H(Q_1 \vert W) + (1-\alpha) H(Q_2 \vert W) \,.
\end{equation}
\end{proposition}

\begin{proof}
Since under $\mathcal{W}_N$ the random variables $(\sigma_j[0,T])_{j=1}^N$ are i.i.d., by Sanov's Theorem (see Theorem 3.2.17 in ~\cite{DeSt89}) we can deduce that the sequence $\{\mathcal{W}_{N_i}\}_{N_i \geq 1}$, with \mbox{$\mathcal{W}_{N_i}( \cdot ) = W^{\otimes N_i} \{\rho_i \in \, \cdot\}$}, satisfies a large deviation principle with rate function $H(\cdot \vert W)$, for $i=1, 2$. Moreover, because of independence, if we consider $A,B \in \mathcal{B}(\mathscr{D}[0,T])$ Borelian sets, we get
\[
\mathcal{W}_N  \{\boldsymbol{\rho}_N \in A \times B\} 
%
= \mathcal{W}_{N_1} \{\rho_1 \in A\} \, \mathcal{W}_{N_2} \{\rho_2 \in B\}.
\]
Therefore, taking closed sets $C_1, C_2 \in \mathcal{B}(\mathscr{D}[0,T])$, we can estimate
\begin{align*}
\limsup_{N \to +\infty} \frac{1}{N} \log \mathcal{W}_N (C_1 \times C_2) &\leq \alpha \limsup_{N_1 \to +\infty} \frac{1}{N_1} \log \mathcal{W}_{N_1} (C_1) + (1-\alpha) \limsup_{N_2 \to +\infty} \frac{1}{N_2} \log \mathcal{W}_{N_2} (C_2)\\
&\leq - \inf_{\substack{Q_1 \in C_1 \\ Q_2 \in C_2}} \left[ \alpha H(Q_1 \vert W) + (1-\alpha) H(Q_2 \vert W)\right] \,.
\end{align*}
The LDP lower bound is proved similarly.
\end{proof}

We are now ready to prove the large deviation principle for the laws $\{\mathcal{P}_N\}_{N \geq 1}$ of $\boldsymbol{\rho}_N$ (under $P_N$) stated in Theorem \ref{thm:LDP}. Thanks to Lemma \ref{lmm:RND_PW} we have identified the Radon-Nikodym derivative that relates $W^{\otimes N}$ and $P_N$. A natural way to develop a large deviation principle is now to rely on Proposition \ref{prop:LDP_RN}. Indeed, by using Lemma \ref{lmm:RND_PW}, we have
\[
\mathcal{P}_N(\cdot) = \int \mathds{1}_{\{\boldsymbol{\rho}_N \in \; \cdot\}} \, dP_N  \nonumber\\
                =\int \exp [NF(\boldsymbol{\rho}_N)] \, \mathds{1}_{\{\boldsymbol{\rho}_N \in \; \cdot\}} \, dW^{\otimes N} \nonumber\\
                = \int \exp [NF(\boldsymbol{Q})] \, \mathds{1}_{\{\boldsymbol{Q} \in \; \cdot\}} \, \mathcal{W}_N(d\boldsymbol{Q}) \,,
\]                
with $\boldsymbol{Q} = \boldsymbol{\rho}_N$. This means that
\begin{equation}\label{DerRNfin}
\frac{d\mathcal{P}_N}{d\mathcal{W}_N}(\boldsymbol{Q}) = \exp [NF(\boldsymbol{Q})] \,.
\end{equation}
Equation \eqref{DerRNfin} allows us to apply Varadhan's Lemma (see Theorem 2.2 in ~\cite{Var84}) and conclude that the LDP for $\{\mathcal{W}_N\}_{N \geq 1}$ with rate function $I(\boldsymbol{Q})$ implies the LDP for $\{\mathcal{P}_N\}_{N \geq 1}$ with rate function $I(\boldsymbol{Q})-F(\boldsymbol{Q})$, where $I(\cdot)$ and $F(\cdot)$ are defined by equations \eqref{RateFct_RN} and \eqref{F}, respectively.

\subsection{Proof of Theorem \ref{thm:MKV}}

We need to define a new process and to derive a technical lemma related to it. \\
Given $\boldsymbol{Q} \in \mathscr{M}_1(\mathscr{D}[0,T]) \times \mathscr{M}_1(\mathscr{D}[0,T])$, we can associate with $\boldsymbol{Q}$ a pair $\left( \eta^{(1)}(t), \eta^{(2)}(t) \right)_{t \in [0,T]}$ of independent Markov processes both with state space $\mathscr{S}$, initial distribution $\lambda$ and with respective time-dependent infinitesimal generators $\mathcal{L}_1^{\boldsymbol{Q}}$ and $\mathcal{L}_2^{\boldsymbol{Q}}$, acting on functions from $\mathscr{S}$ to $\mathbb{R}$ as 
\begin{align*}
\mathcal{L}_1^{\boldsymbol{Q}} f \left( \eta \right) &= e^{-\eta \left[ R_1 \left( \mathbf{m}_{\boldsymbol{Q}(t)} \left( \eta \right) \right) + h_1 \right]} \, \nabla^{\eta} f \left( \eta \right) \\
\mathcal{L}_2^{\boldsymbol{Q}} f \left( \eta \right) &= e^{-\eta \left[ R_2 \left( \mathbf{m}_{\boldsymbol{Q}(t)} \left( \eta \right) \right) + h_2 \right]} \, \nabla^{\eta} f \left( \eta \right) .
\end{align*}
Let $P^{\boldsymbol{Q}}_i$ be the distribution of the process $\left( \eta^{(i)}(t) \right)_{t \in [0,T]}$, with $i=1, 2$, then the joint law of the pair $\left( \eta^{(1)}(t), \eta^{(2)}(t) \right)_{t \in [0,T]}$ is denoted by $\boldsymbol{P}^{\boldsymbol{Q}} = P^{\boldsymbol{Q}}_1 \otimes P^{\boldsymbol{Q}}_2$. \\
Next proposition shows a very important property of $\boldsymbol{P}^{\boldsymbol{Q}}$.

\begin{proposition}
For every $\boldsymbol{Q} \in \mathscr{M}_1(\mathscr{D}[0,T]) \times \mathscr{M}_1(\mathscr{D}[0,T])$ such that $\mathcal{I}(\boldsymbol{Q}) < +\infty$,
\begin{equation}\label{mfI(Q)}
\mathcal{I}(\boldsymbol{Q}) = \alpha H \left( Q_1 \left\vert P^{\boldsymbol{Q}}_1 \right. \right) + (1-\alpha) H \left( Q_2 \left\vert P^{\boldsymbol{Q}}_2 \right. \right) \,.
\end{equation}
\end{proposition}

\begin{proof}
Similarly to the proof of Lemma \ref{lmm:RND_PW}, by Girsanov's formula for Markov processes and reversibility, it is easy 
to verify that the following representations for $F_1(Q_1)$ and $F_2(Q_2)$ (defined in \eqref{F1-2}) hold
\[
F_i(Q_i) = \int dQ_i \left( \eta^{(i)}[0,T] \right) \log \frac{dP^{\boldsymbol{Q}}_i}{dW} \left( \eta^{(i)}[0,T] \right)
\quad \mbox{ for } i=1, 2.
\] 
By combining what we obtained, we can compute
\begin{align*}
\mathcal{I}(\boldsymbol{Q}) &= \alpha \left[ H(Q_1 \vert W) - F_1(Q_1) \right] + (1-\alpha) \left[ H(Q_2 \vert W) - F_2(Q_2) \right] \\
&= \alpha \left[ \int dQ_1 \log \frac{dQ_1}{dW} - \int dQ_1 \log \frac{dP^{\boldsymbol{Q}}_1}{dW} \right] + (1-\alpha) \left[ \int dQ_2 \log \frac{dQ_2}{dW} - \int dQ_2 \log \frac{dP^{\boldsymbol{Q}}_2}{dW} \right]\\
&=\alpha H \left(Q_1 \left\vert P^{\boldsymbol{Q}}_1 \right. \right) + (1-\alpha) H \left(Q_2 \left\vert P^{\boldsymbol{Q}}_2 \right. \right)\,.
\end{align*}
\end{proof}

We can now conclude the proof of Theorem \ref{thm:MKV}. \\
By properness of the relative entropy, from \eqref{mfI(Q)} we have that $\mathcal{I}(\boldsymbol{Q}) = 0$ is equivalent to $\boldsymbol{Q}=\boldsymbol{P^Q}$.  Let us suppose $\boldsymbol{Q}^*$ is a solution of this last equation. Then, in particular, $\boldsymbol{q}(t) = \boldsymbol{Q}^*(t) = \boldsymbol{P}^{\boldsymbol{Q}^*}(t)$. The marginals of a Markov process are solutions of the corresponding forward equation that, in this case, leads to the fact that $\boldsymbol{q}(t)$ is a solution of \eqref{MKV1}. This differential equation, being an equation in finite dimension with locally Lipschitz coefficients, has at most one solution in $[0,T]$. Since $\boldsymbol{P}^{\boldsymbol{Q}^*}$ is totally determined by the flow $\boldsymbol{q}(t)$, it follows that equation $\boldsymbol{Q} = \boldsymbol{P^Q}$ has at most one solution. The existence of a solution derives from the fact that $\mathcal{I}(\boldsymbol{Q})$ is the rate function of a large deviation principle and therefore it must have at least one zero. \\
It remains to prove the convergence of $\boldsymbol{\rho}_N$ (law of large numbers). Let $O$ be an arbitrary open neighborhood of $\boldsymbol{Q}^*$ in $\mathscr{M}_1(\mathscr{D}[0,T]) \times \mathscr{M}_1(\mathscr{D}[0,T])$. By the lower semi-continuity of $\mathcal{I}(\cdot)$ and the compactness of its level-sets a standard argument shows that $K (O) := \inf_{\boldsymbol{Q} \notin O} \mathcal{I}(\boldsymbol{Q}) > 0$. By the large deviation upper bound in Theorem \ref{thm:LDP} (type \eqref{LDPc}) there exists a positive constant $A$ such that 
\[
\mathcal{P}_N (\boldsymbol{\rho}_N \notin O) \leq A e^{-N K(O)} \,,
\] 
thus giving convergence to zero with exponential decay.
\subsection{Proof of Theorem \ref{thm:PhDia}}

Let $V \left( \mathbf{x} \right) = \left( 2 \sinh \left[ R_1 \left( \mathbf{x} \right) \right] - 2 x_1 \cosh \left[ R_1 \left( \mathbf{x} \right) \right] ,  2 \sinh \left[ R_2 \left( \mathbf{x} \right) \right] - 2 x_2 \cosh \left[ R_2 \left( \mathbf{x} \right) \right] \right)$. Under our assumptions, we must discuss existence and linear stability of equilibrium solutions for the differential system $\dot{\mathbf{m}}(t) = V \left( \mathbf{m}(t) \right)$.

\emph{Existence of stationary solutions.} Denote by $\mathbf{m}^*$ a pair such that $V \left( \mathbf{m}^* \right)=\boldsymbol{0}$. Any stationary point is then implicitly defined by
\begin{equation}\label{MKV:StatSols}
\begin{array}{l}
m_1^* = \tanh \left[ \alpha J_{11} m_1^* + (1-\alpha) J_{12} m_2^* \right] \\
m_2^* = \tanh \left[ \alpha J_{12} m_1^* + (1-\alpha) J_{22} m_2^* \right].
\end{array}
\end{equation}
By inverting equations \eqref{MKV:StatSols} when $J_{12} \neq 0$ (i.e. the model does not degenerate into two independent Curie-Weiss models), we get
\begin{align}
m_2^* &= \frac{1}{(1-\alpha) J_{12}} \left[ \mathrm{arctanh} \left( m_1^* \right) - \alpha J_{11} m_1^* \right] \tag{$\gamma_1$}\\
m_1^* &= \frac{1}{\alpha J_{12}} \left[ \mathrm{arctanh} \left( m_2^* \right) - (1-\alpha) J_{22} m_2^* \right]. \tag{$\gamma_2$}
\end{align}
The points that verify simultaneously the two equations are the intersections of the curve $\gamma_1 \left( m_1^* \right)$, the graph of the function $m_2^* \left( m_1^* \right)$, with the curve $\gamma_2 \left( m_2^* \right)$, the graph of the function $m_1^* \left( m_2^* \right)$. To detect such points we use a graphical approach.\\
Observe that it suffices to consider $J_{12} > 0$, since the complementary  case can be obtained by symmetry about the vertical axis because of the oddness of the involved functions. From now on we restrict to this range of the parameter $J_{12}$.\\
A quick analysis shows the curves $\gamma_1$ and $\gamma_2$ are characterized by the following features (meant on the $(m_1^*, m_2^*)$ plane for $\gamma_1$, whilst on the $(m_2^*, m_1^*)$ one for $\gamma_2$):
\begin{itemize}
\item they are defined over $[-1,+1]$ and they are odd functions;
\item they diverge to $\pm \infty$ as the corresponding independent variable tends to $\pm 1$;
\item  
\begin{description}
\item[$(\gamma_1)$] 
\begin{enumerate}[label=(\Roman*)]
\item if $\alpha J_{11} \leq 1$, then $\gamma_1$ is increasing for all values $m_1^* \in [-1,+1]$;
\item if $\alpha J_{11} > 1$, then $\gamma_1$ is increasing for $m_1^* \leq  -\sqrt{1-\frac{1}{\alpha J_{11}}}$ or for $m_1^*  \geq \sqrt{1-\frac{1}{\alpha J_{11}}}$. Moreover, $m_1^* = -\sqrt{1-\frac{1}{\alpha J_{11}}}$ and $m_1^* = \sqrt{1-\frac{1}{\alpha J_{11}}}$ correspond to a local maximum and minimum for $\gamma_1$, respectively.
\end{enumerate}
\item[$(\gamma_2)$]
\begin{enumerate}[label=(\Roman*)]
\item if $(1-\alpha) J_{22} \leq 1$, then $\gamma_2$ is increasing for all values $m_2^* \in [-1,+1]$;
\item if $(1-\alpha) J_{22} > 1$, then $\gamma_2$ is increasing for $m_2^* \leq  -\sqrt{1-\frac{1}{(1-\alpha) J_{22}}}$ or for $m_2^*  \geq \sqrt{1-\frac{1}{(1-\alpha) J_{22}}}$. Moreover, $m_2^* = -\sqrt{1-\frac{1}{(1-\alpha) J_{22}}}$ and $m_2^* = \sqrt{1-\frac{1}{(1-\alpha) J_{22}}}$ correspond to a local maximum and minimum for $\gamma_2$, respectively.
\end{enumerate}
\end{description}
\item they are convex for positive abscissas and concave for negative ones with a unique inflection point at the origin. 
\end{itemize}
If we put the graphs $\gamma_1$ and $\gamma_2$ (after a ninety degrees clockwise rotation and a reflection about the horizontal axis) in the same $(m_1^*,m_2^*)$ coordinate system, it is possible to understand how many times the two curves intersect depending on the parameters. We must match in all possible ways their shapes, determined by conditions $(\gamma_1-I\&II)$ and $(\gamma_2-I\&II)$. See Figure~\ref{fig:3} for an example. This leads to the number of solutions of \eqref{MKV:StatSols}.\\

\setcounter{subfigure}{0}
\begin{figure}
\centering%
\subfigure[Case $\alpha J_{11} \leq 1$ and $(1-\alpha) J_{22} \leq 1$.]{
\includegraphics[width=0.4\textwidth]{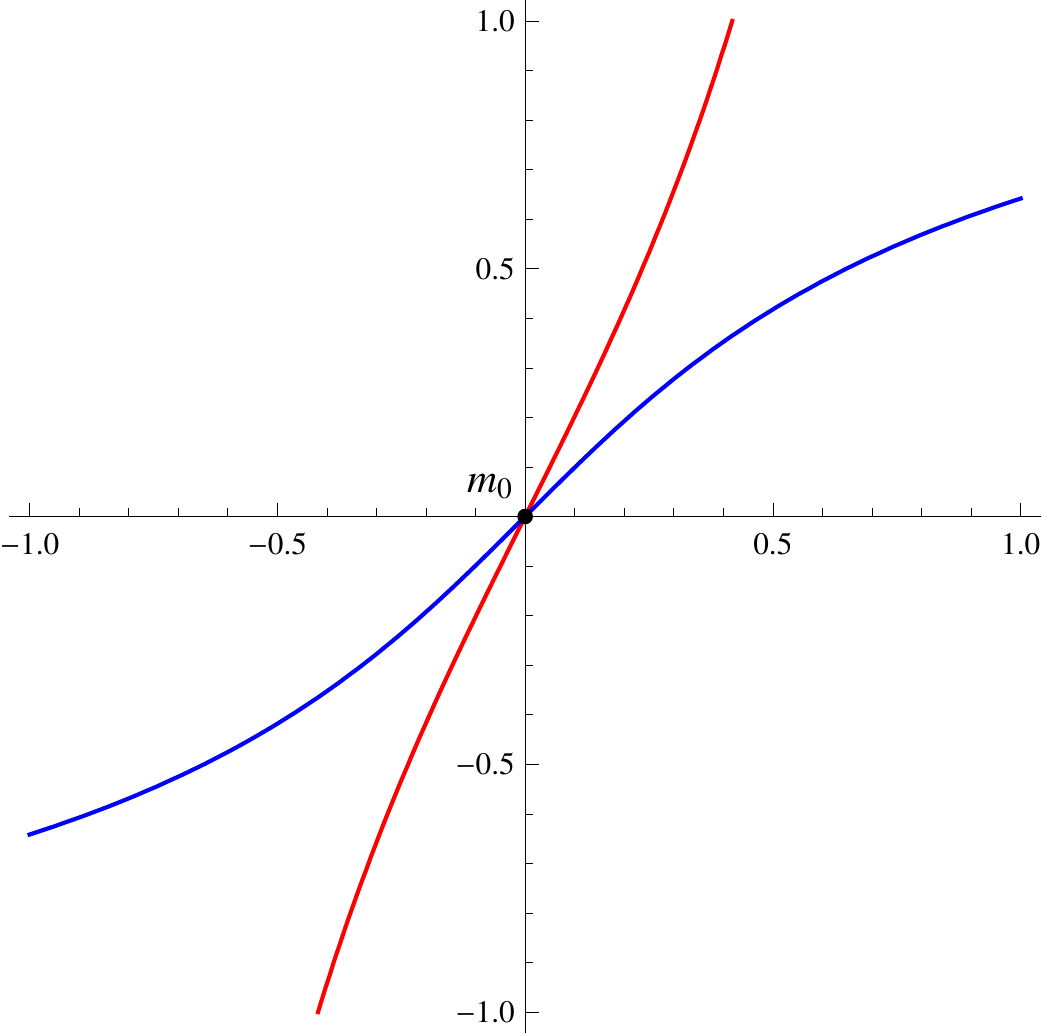} 
\includegraphics[ width=0.4\textwidth]{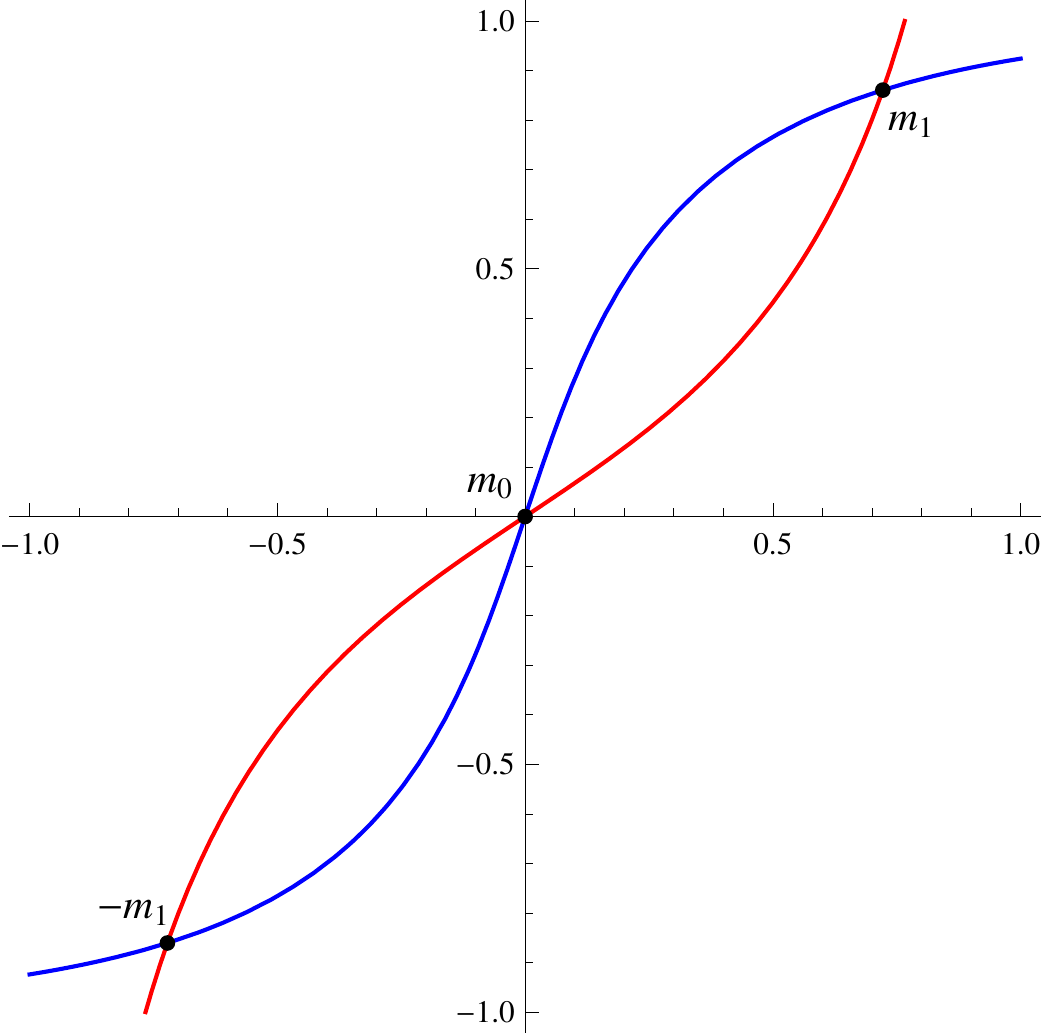} 
}
\subfigure[Case $\alpha J_{11} \leq 1$ and $(1-\alpha) J_{22} > 1$ (left panel) or, conversely, case $\alpha J_{11} > 1$ and $(1-\alpha) J_{22} \leq 1$ (right panel).]{
\includegraphics[width=0.4\textwidth]{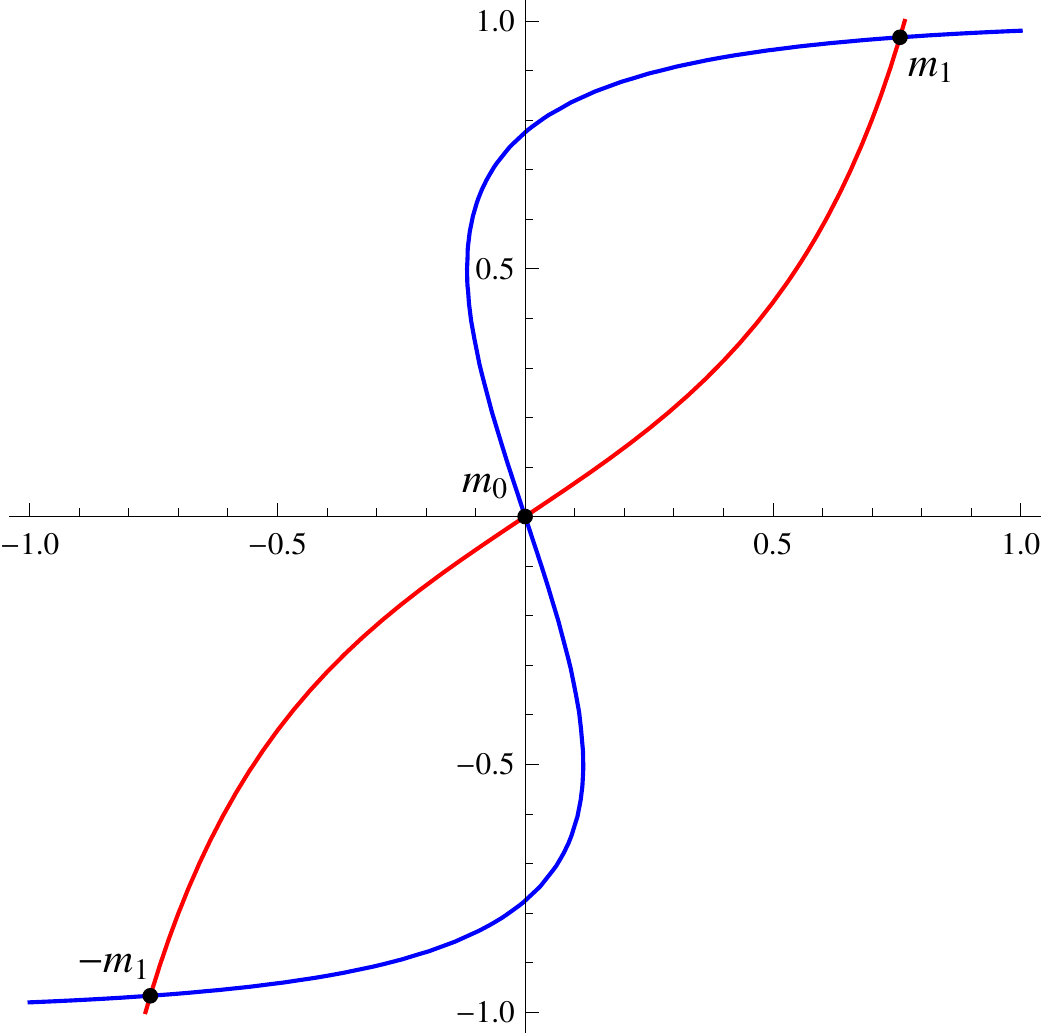} 
\includegraphics[ width=0.4\textwidth]{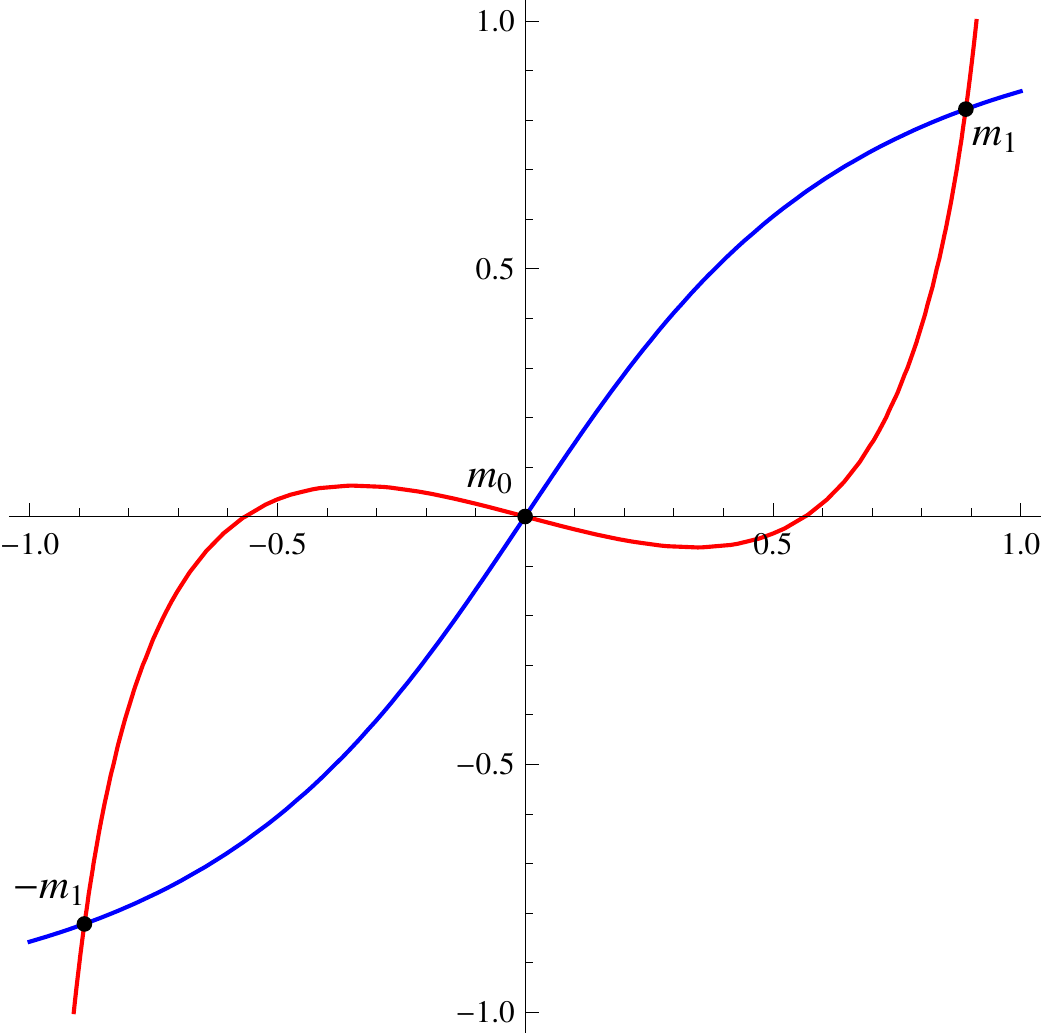} 
}
\centerline{
\subfigure[Case $\alpha J_{11} > 1$ and $(1-\alpha) J_{22} > 1$.]{
\includegraphics[width=0.4\textwidth]{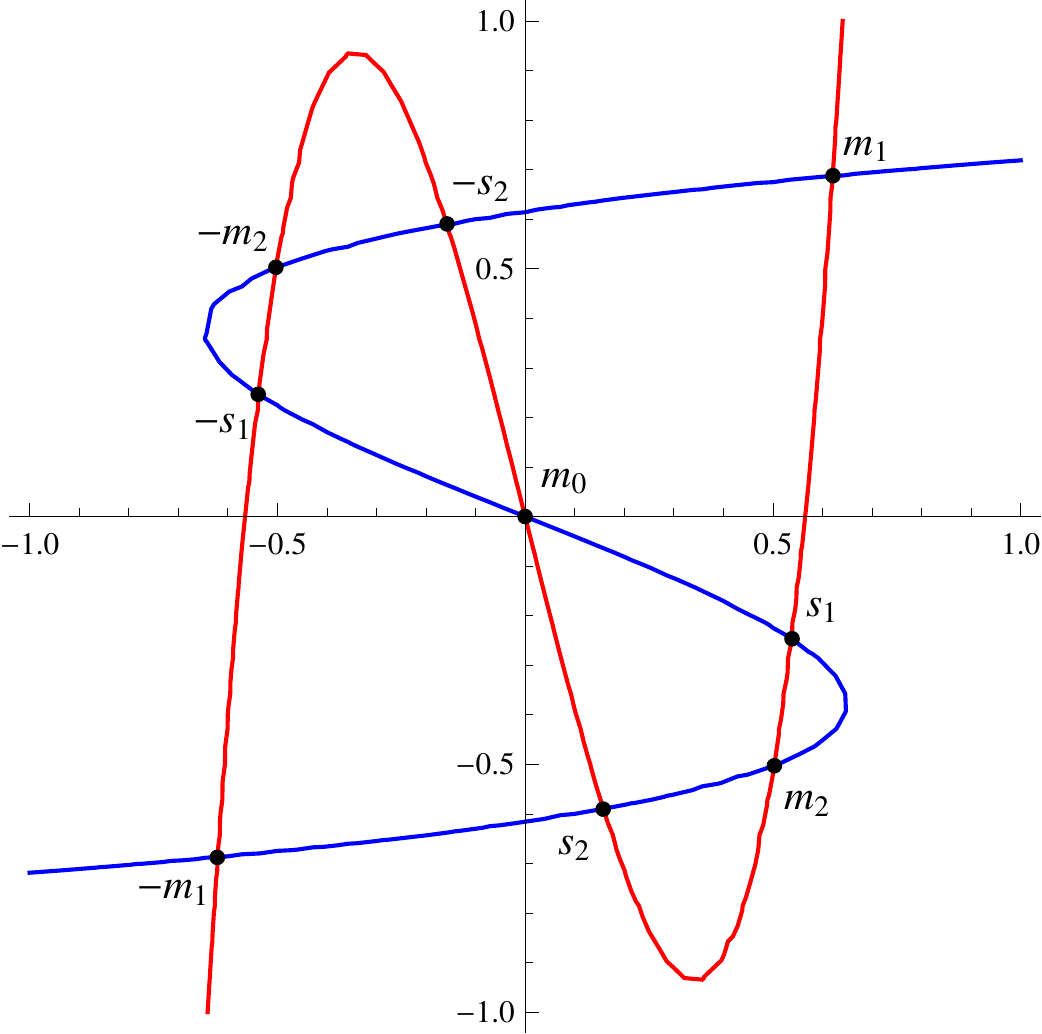} 
\includegraphics[ width=0.4\textwidth]{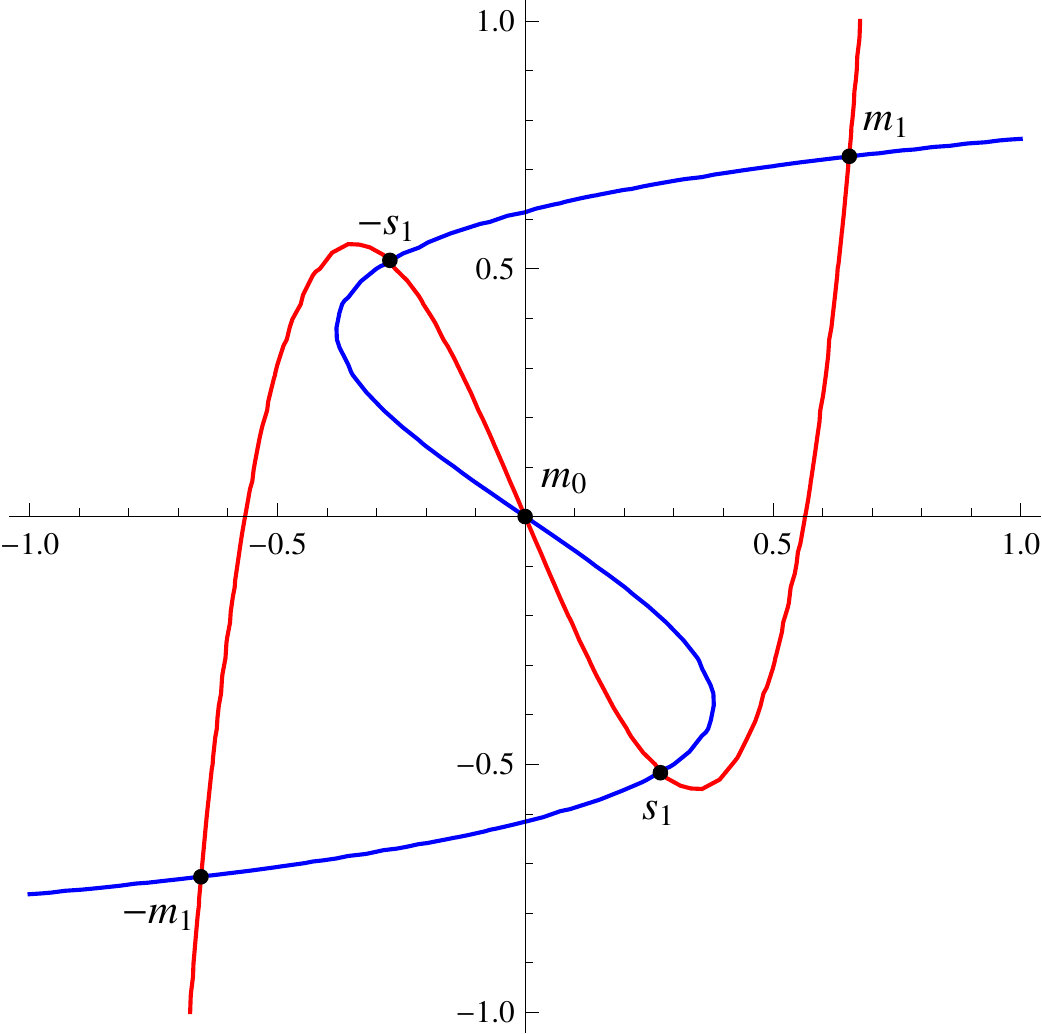} 
\includegraphics[ width=0.4\textwidth]{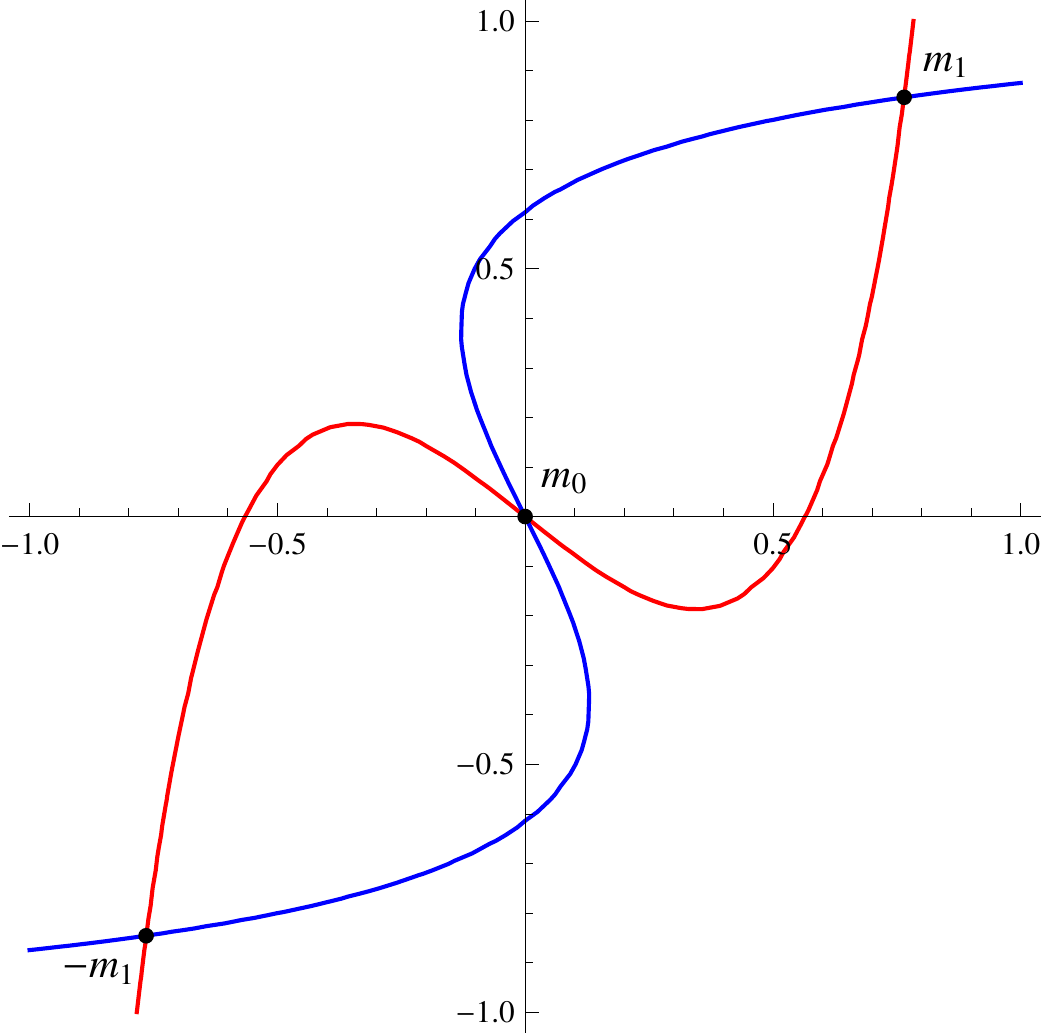} 
}}
\caption{Plot of the curves $\gamma_1$ ({\red \textbf{---}}) and $\gamma_2$ ({\blue \textbf{---}}) for different values of the parameters.}
\label{fig:3}
\end{figure}

Now we focus on the critical thresholds for the transitions \emph{a(i)--a(ii)}, \emph{c(i)--c(ii)} and \emph{c(ii)--c(iii)} in the statement of Theorem~\ref{thm:PhDia}. Before to continue we introduce the following notation: here and in the sequel of the proof, by abuse of notation, we will denote by $\gamma_2^{-1} (m_1^*)$ a local inverse of $\gamma_2(m_2^*)$; we intend $\gamma_2^{-1}$ to be defined over a suitable monotonicity interval that may change from case to case. 

\begin{description}[font=\normalfont]
\item[Transition in \emph{(a)}.] A simple convexity argument shows that to know the number of intersections between $\gamma_1$ and $\gamma_2$ we must control the slope of the two curves at the origin.   In particular, since $\gamma_1'(0) = \frac{1 - \alpha J_{11}}{(1-\alpha) J_{12}}$ and $\left(\gamma_2^{-1} \right)'(0) = \frac{\alpha J_{12}}{1-(1-\alpha) J_{22}}$, it yields
\begin{itemize}
\item if $\gamma_1'(0) \geq \left( \gamma_2^{-1} \right)'(0) \Leftrightarrow J_{12} \leq \sqrt{\frac{(1 - \alpha J_{11})(1-(1-\alpha) J_{22})}{\alpha (1-\alpha)}}$, there exists a unique intersection, which is the origin; 
\item if $\gamma_1'(0) < \left( \gamma_2^{-1} \right)'(0) \Leftrightarrow J_{12} > \sqrt{\frac{(1 - \alpha J_{11})(1-(1-\alpha) J_{22})}{\alpha (1-\alpha)}}$, beyond the origin, two further symmetric intersections $\pm \mathbf{m}^{(1)}$ arise.
\end{itemize}
\item[Transitions in \emph{(c)}.] Observe that $\gamma_1' (0)$ is an increasing function of $J_{12}$. On the contrary, $\left( \gamma_2^{-1} \right)'(0)$ is decreasing with respect to the same quantity. Therefore, by convexity again we can easily conclude that, for $J_{12} \geq J_{12}^{(c)}$, $\gamma_1$ and $\gamma_2$ must intersect three times. The transition \emph{c(ii)--c(iii)} remains thus proved.\\
We deal now with the passage \emph{c(i)--c(ii)}. For the sake of clarity and to be able to  refer to Figure~\ref{fig:1}, without loss of generality, we fix $(1-\alpha) J_{22} > 1$. When $J_{12} < J_{12}^{(c)}$, we may have five or nine intersections depending on the relative height of wells of the curves (monotonically tuned by the parameter $J_{12}$). In particular, the boundary between the two phases is determined by the tangency condition between $\gamma_1$ and $\gamma_2$. Indeed, $\widetilde{J}_{12}^{(c)}$ is the value of $J_{12}$ such that the following system is satisfied:
\[
\left\{
\begin{array}{l}
m_1^* = \tanh \left[ \alpha J_{11} m_1^* + (1-\alpha) J_{12} m_2^* \right] \\
m_2^* = \tanh \left[ \alpha J_{12} m_1^* + (1-\alpha) J_{22} m_2^* \right] \\
J_{12} = \sqrt{\dfrac{\left( \frac{1}{1- (m_1^*)^2} - \alpha J_{11} \right) \left( \frac{1}{1- (m_2^*)^2} - (1-\alpha) J_{22} \right)}{\alpha (1-\alpha)}} \,.
\end{array}
\right.
\]
Therefore, $\widetilde{J}_{12}^{(c)}$ is a continuous function of $\alpha J_{11}$, whose graph bifurcates from $\alpha J_{11} = 1$ and is defined over $]1,+\infty[$. This implies that, for every fixed value $\alpha J_{11}$, by increasing $J_{12}$ from $0$,  the system always undergoes the transitions from $9$ to $5$ and from $5$ to $3$ stationary solutions.
\end{description}

\emph{Linear stability.}  The matrix of the system linearized around a stationary solution $\mathbf{m}^*$ is given by
{\small
\[
DV \left( \mathbf{m}^* \right) =
\begin{pmatrix}
2 \left[ \alpha J_{11} \left[ 1-\left(m_1^* \right)^2 \right] -1\right] \cosh \left[ R_1 \left(  \mathbf{m}^* \right) \right] & 2 (1-\alpha) J_{12} \left( 1-\left(m_1^* \right)^2\right) \cosh \left[  R_1\left( \mathbf{m}^* \right) \right] \\
2 \alpha J_{12} \left( 1- \left( m_2^* \right)^2 \right) \cosh \left[  R_2\left(  \mathbf{m}^* \right) \right] & 2 \left[ (1-\alpha) J_{22} \left( 1- \left( m_2^* \right)^2 \right) -1 \right] \cosh \left[  R_2\left(  \mathbf{m}^* \right) \right]
\end{pmatrix} ,
\]
}%
with eigenvalues
{\small
\begin{multline*}
\lambda_{\pm} \left( \mathbf{m}^* \right) =  \left[ \alpha J_{11} \left( 1- \left(m_1^* \right)^2 \right) -1\right] \cosh \left[  R_1\left( \mathbf{m}^* \right) \right] + \left[ (1-\alpha) J_{22} \left( 1- \left( m_2^* \right)^2 \right) -1 \right] \cosh \left[  R_2\left( \mathbf{m}^* \right) \right] \\
\pm \bigg\{ \left\{ \left[ \alpha J_{11} \left( 1- \left(m_1^* \right)^2 \right) -1\right] \cosh \left[  R_1\left( \mathbf{m}^* \right) \right] - \left[ (1-\alpha) J_{22} \left( 1- \left( m_2^* \right)^2 \right) -1 \right] \cosh \left[  R_2\left( \mathbf{m}^* \right) \right] \right\}^2 \\
+ 4 \alpha (1-\alpha) J_{12}^2 \left( 1- \left(m_1^* \right)^2 \right) \left( 1- \left( m_2^* \right)^2 \right)  \cosh \left[  R_1\left( \mathbf{m}^* \right) \right] \cosh \left[  R_2\left( \mathbf{m}^* \right) \right] \bigg\}^{1/2} \,.
\end{multline*}
}%
Note that $\lambda_{\pm} \left( \mathbf{m}^* \right) \in \mathbb{R}$ for all values of the parameters and that clearly $\lambda_- \left( \mathbf{m}^* \right) < \lambda_+ \left( \mathbf{m}^* \right)$. Moreover, if $\alpha J_{11} \leq 1$ and $(1-\alpha) J_{22} \leq 1$ (or, analogously, $\alpha J_{11} > 1$ and $(1-\alpha) J_{22} > 1$), a direct substitution shows that for $J_{12} = J_{12}^{(c)}$ it holds $\lambda_{+} \left( \mathbf{m}^{(0)} \right) = 0$, meaning the eigenspace relative to $0$ is a neutral direction for the linearized system.\\
It remains to study the spectrum of $DV$ at the stationary points of equation \eqref{MKV2} in all other ranges of parameters. The analysis consists of standard straightforward manipulations. Therefore, for the sake of brevity, we skip the algebraic details giving only the key steps the procedure relies on. \\
First observe that, by using the same notation and convention we introduced in the previous paragraph, the eigenvalues $\lambda_{\pm} (\mathbf{m}^*)$ may be rewritten as
{\small
\begin{multline*}
\hspace{-0.3cm} \lambda_{\pm} \left( \mathbf{m}^* \right) =  - (1-\alpha) J_{12} \left( 1- \left(m_1^* \right)^2 \right) \gamma_1' \left(m_1^* \right) \cosh \left[  R_1\left( \mathbf{m}^* \right) \right] - \alpha J_{12} \left( 1- \left( m_2^* \right)^2 \right) \left[ \left( \gamma_2^{-1}\right)' \left( m_1^* \right) \right]^{-1} \cosh \left[  R_2\left( \mathbf{m}^* \right) \right] \\
\pm \Bigg\{ \left\{ -(1-\alpha) J_{12} \left( 1- \left(m_1^* \right)^2 \right) \gamma_1' \left(m_1^* \right) \cosh \left[  R_1\left( \mathbf{m}^* \right) \right] + \alpha J_{12} \left( 1- \left( m_2^* \right)^2 \right) \left[ \left( \gamma_2^{-1}\right)' \left( m_1^* \right) \right]^{-1} \cosh \left[  R_2\left( \mathbf{m}^* \right) \right] \right\}^2 \\
+ 4 \alpha (1-\alpha) J_{12}^2 \left( 1- \left(m_1^* \right)^2 \right) \left( 1- \left( m_2^* \right)^2 \right)  \cosh \left[  R_1\left( \mathbf{m}^* \right) \right] \cosh \left[  R_2\left( \mathbf{m}^* \right) \right] \Bigg\}^{1/2} \,.
\end{multline*}
}%
Thus, to determine their sign at $\mathbf{m}^*$ it is crucial the knowledge of the behavior of the curves $\gamma_1$ and $\gamma_2$ in a neighborhood of such critical point. In particular, if the equilibrium $\mathbf{m}^*$ is such that

\begin{itemize}
\item $\gamma_1'(m_1^*) > 0$, $\left( \gamma_2^{-1} \right)' (m_1^*) > 0$ and $\gamma_1'(m_1^*) > \left( \gamma_2^{-1} \right)' (m_1^*)$; then, we get $\lambda_{\pm} (\mathbf{m}^*) < 0$ and therefore $\mathbf{m}^*$ is linearly stable. 
\item $\gamma_1'(m_1^*) < 0$, $\left( \gamma_2^{-1} \right)' (m_1^*) < 0$ and $\gamma_1'(m_1^*) < \left( \gamma_2^{-1} \right)' (m_1^*)$; then, we get $\lambda_{\pm} (\mathbf{m}^*) > 0$ and therefore $\mathbf{m}^*$ is linearly unstable. 
\item $\gamma_1'(m_1^*) < 0$, $\left( \gamma_2^{-1} \right)' (m_1^*) > 0$ (or, conversely, $\gamma_1'(m_1^*) > 0$, $\left( \gamma_2^{-1} \right)' (m_1^*) < 0$); then, we get \mbox{$\lambda_{+} (\mathbf{m}^*) \lambda_{-} (\mathbf{m}^*) < 0$} and therefore $\mathbf{m}^*$ is a saddle for the linearized system. 
\end{itemize}
By checking the stationary solutions in the statement of Theorem~\ref{thm:PhDia}, labeled as in Figure~\ref{fig:3}, the conclusion of the proof follows.

\section{Conclusions}

In this paper we have investigated a simple multi-species interacting spin system whose intensive study has been motivated by strong and concrete applicability in social sciences. Unlike the existing literature, our work has proposed a dynamic version of the model. The main features are
\begin{itemize}
\item
Mean field type interaction: the particles lie on the complete graph and thus each of them interact with all the others.
\item
Presence of two types of spins: the particles are divided into two reference groups by differentiating the interactions between each other. In particular, we have
\begin{itemize}
\item two distinct \emph{intra-group interactions}, tuning the interaction strength between particles of the same population;
\item a \emph{inter-groups interaction}, controlling the influence between spins of different populations. 
\end{itemize} 
\item
Time-reversible Markovian microscopic dynamics.
\item
Order parameter: the pair of group magnetizations is a sufficient statistic for the model.
\end{itemize}
We have shown an asymptotic result in the limit as the number of particles goes to infinity. We have proved a law of large numbers, that describes the dynamics of a typical spin in the limit of an infinite size system. The large time behavior of this dynamics exhibits phase transition: depending on the parameters of the model, and possibly on the initial condition, states of populations may show different degrees of polarization - due to strongly influencing connections - or tend to a ``neutral'' configuration, when the coupling within sites is too weak. \\
An important and interesting further step would be to determine how macroscopic observables fluctuate around their mean values when the system is put out of or at the critical point. This issue is left for future research.

\section{Acknowledgments}

The author thanks Dr. Alessandra Bianchi and Prof. Pierluigi Contucci for valuable comments and fruitful discussions. This work was supported by the FIRB research grant RBFR10N90W. 

\bibliographystyle{alpha} 
\newcommand{\etalchar}[1]{$^{#1}$}


\end{document}